\documentclass[1p,twoside]{elsarticle}
\usepackage{amssymb}
\usepackage{amsmath}
\usepackage{graphicx}
\usepackage{youngtab}
\usepackage{stmaryrd}
\usepackage{float}
\usepackage{pict2e,color}
\usepackage{epsfig}
\usepackage[hang]{subfigure}
\usepackage{lineno,hyperref}
\modulolinenumbers[5]

\newtheorem{theorem}{Theorem}
\newtheorem{lemma}{Lemma}
\newtheorem{corollary}{Corollary}
\newtheorem{proposition}{Proposition}

\newproof{proof}{Proof}

\newdefinition{definition}{Definition}
\newdefinition{example}{Example}

\newcommand{\hookdownarrow}{\mathrel{\rotatebox[origin=c]{-90}{$\hookrightarrow$}}}

\newcommand{\sgn}{\mathop{\rm sgn}}
\newcommand{\Hom}{\mathop{\rm Hom}\nolimits}
\newcommand{\Z}{\mathbb Z}
\newcommand{\Q}{\mathbb Q}

\newcommand{\C}{\mathbb C}

\makeatletter
\@namedef{subjclassname@2020}{\textup{2020} Mathematics Subject Classification}
\def\@journal{arXiv}\date{}
\makeatother

\begin{document}


\title{Subset Representations and Eigenvalues of the Universal Intertwining Matrix}

\author{Dirk Siersma}
\ead{D.Siersma@uu.nl}
\author{Wilberd van der Kallen\corref{cor1}} 
\ead{W.vanderKallen@uu.nl}
\address{Mathematical Institute, Utrecht University,  P.O.Box 80.010,\\ 3508 TA Utrecht, The Netherlands}
\cortext[cor1]{Corresponding author}


\begin{abstract}
We solve a combinatorial question concerning eigenvalues of the universal intertwining endomorphism of a subset representation.
\end{abstract}

\maketitle

\section{Introduction}
The symmetric group $S_n$ acts on the set $\mathcal{C}^{n}_{k}$ of subsets  of $k$ elements of a set of $n$ elements. This defines  a
representation, known as a subset representation. Subset representations are related to Young Tableaux  with 2 rows. 
We consider  the universal intertwining matrix $B_{(n-k,k)}$ for a subset representation and use Schur's lemma and Young's rule to show that 
the eigenvalues are $\Z$-linear in the natural parameters of the intertwining matrix  (Proposition~\ref{Prop 1}).
Next we compute the eigenvalues (Theorem \ref{Thm 1}).
In the terminology that is customary in algebraic combinatorics, what we are doing is recomputing the ``eigenmatrix $P$ of a Johnson scheme''.
This eigenmatrix was determined much earlier by Delsarte~\cite{De1}. (Around this time Philips was developing the compact disc.)
Theorem \ref{Thm 1} is then applied to justify the evaluation  of the Eisenbud-Levine-Khimshiashvili (ELK) signature formula
for the gradient index at a degenerate star in \cite{Si}. For this we also need the package MultiSum \cite{weg}, in order to perform a summation
of complicated hypergeometric terms.

\section{Subset representations, intertwiners, Young's rule}

\subsection{The question}

\noindent Let $n$ be a positive integer and $0\le k\le \lfloor n/2 \rfloor$.
Consider combinations $\mathcal{C}^{n}_{k}$ of $k$ elements (unordered) out of a set of $n$ elements. 
Take an arbitrary tuple of complex numbers $b_0,\cdots ,b_{k}$. We constitute a matrix $B=B_{(n-k,k)}$, 
where the rows and columns are indexed (in lexicographic order) by elements of $\mathcal{C}^{n}_k$.
The matrix elements are defined as follows: 
\begin{center}
$ < \sigma, \tau> {}= b_p$ if  $ \sigma \cap \tau $ has $p$ elements ($ 0 \le p \le k$).
\end{center}
We want to compute the eigenvectors and eigenvalues. The result is needed in \cite[\S 3]{Si} for the computation of a `gradient index'.

\begin{figure}
	\centering
\begin{minipage}{0.49\textwidth}
\centering
	\includegraphics[width=0.4 \linewidth]{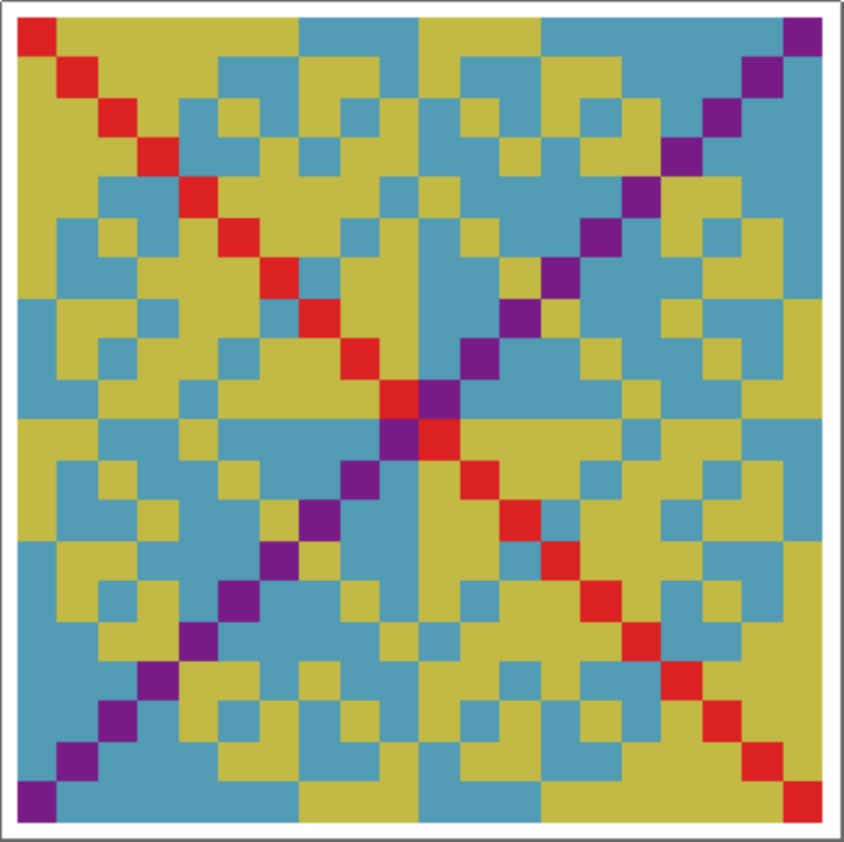}
\end{minipage}
\begin{minipage}{.49\textwidth}
\centering
	\includegraphics[width=0.4 \linewidth]{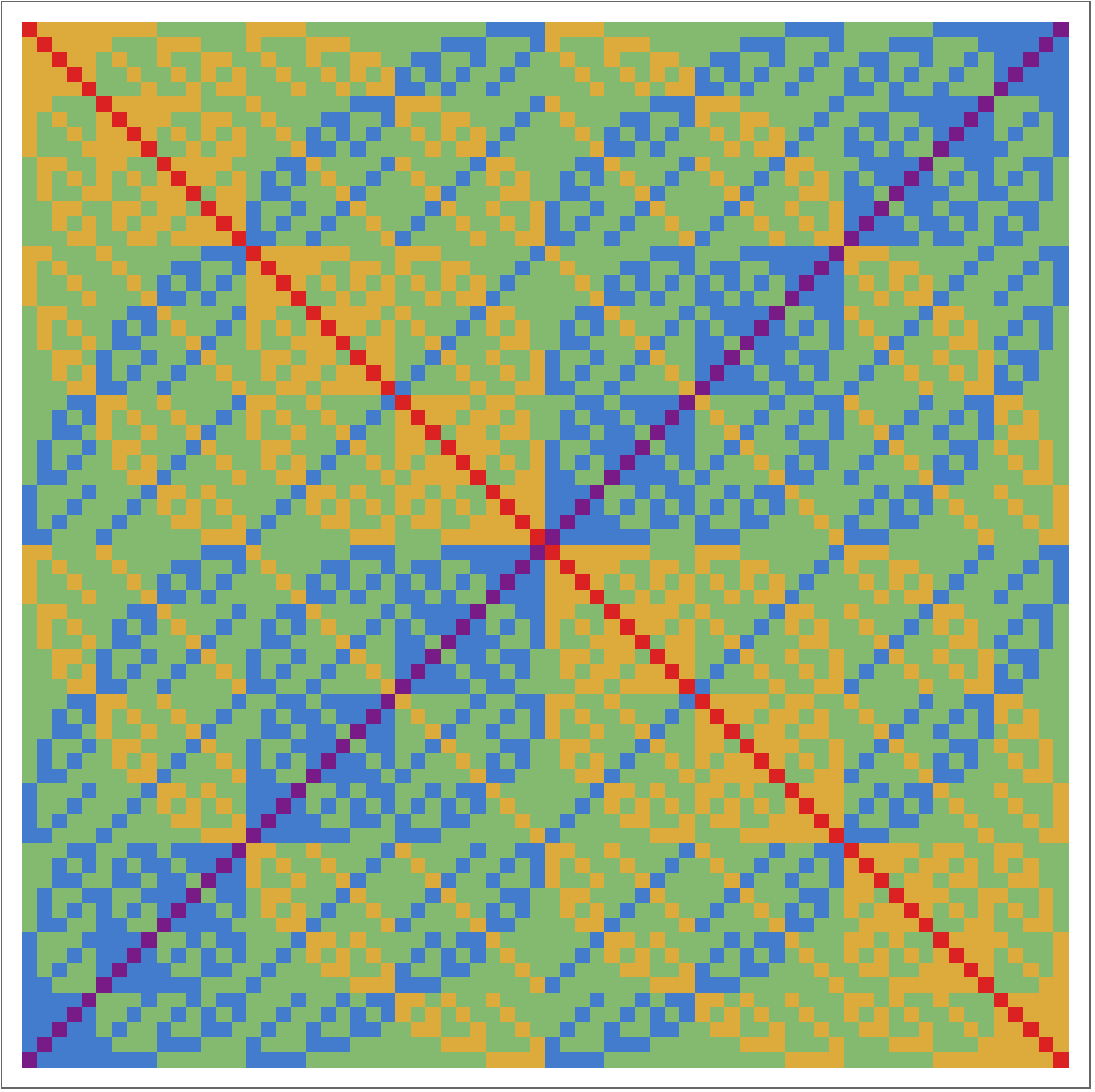}
\end{minipage}
	\label{fig:array}
\caption{$B_{2,2}$ and $B_{3,3}$;  coloured by $b_p$.}
\end{figure}

\subsection{Intertwining}
The symmetric group $S_n$ acts on  $\mathcal{C}^{n}_{k}$  and thus defines  a permutation representation known as $M^{(n-k,k)}$
in \cite[p.\ 86]{Fu}. In Prasad \cite[\S 2.5]{Pr}  
the representation is called a subset representation $\C[X_k]$, where $X_k$ is our $\mathcal{C}^{n}_{k}$. 

A linear map  $N:M^{(n-k,k)}\to M^{(n-k,k)}$ is called {\it intertwining}  if  $g N = N g$  for all $g \in S_n$.

\smallskip

\begin{lemma}Let $N\in \Hom_\C(M^{(n-k,k)},M^{(n-k,k)})$. Use the standard basis of $M^{(n-k,k)}$ to view $N$ as a matrix.
Then $N$ is intertwining if and only if
the matrix elements $N_{\sigma_1,\tau_1}$ and $N_{\sigma_2,\tau_2}$ are equal as soon as 
$\sigma_1 \cap \tau_1$ and  $\sigma_1 \cap \tau_1$  have the same cardinality. In particular, $B$ is the `universal intertwining matrix'.
\end{lemma}

\begin{proof} This is easy  and essentially Theorem 2.51 in Prasad \cite{Pr}.\qed
\end{proof}

\subsection{Specht modules}
With $n$, $k$ as above,
let $\nu$  be the two part partition $(n-k,k)$ of $n$.
We define $T_k$ to be the maximal standard tableau \cite[pp. 84--85]{Fu} of shape $\nu$. 
One could call its numbering lexicographic. See figure \ref{T2} for examples.
   \begin{figure}[H]\caption{$T_1, T_2 ,T_3$  if $n=6$.}\label{T2}
  \begin{center}
   \young(1345,6) ,\qquad \young(1234,56), \qquad  \young(123,456)\ .
   \end{center}
   \end{figure}

Recall that the {\it row subgroup} $R(T_k)$ of $T_k$ is the subgroup of $S_n$ which consists of those permutations that permute the entries
of each row among themselves. Similarly the {\it column subgroup} $C(T_k)$ of $T_k$ is the subgroup of $S_n$ which consists 
of those permutations that permute the entries of each column among themselves. One now puts
 $$a_k=\sum_{p\in R(T_k)} p\ , \qquad b_k=\sum_{q\in C(T_k)}\sgn(q)q\ , \qquad c_k = b_ka_k\ ,$$
where the  product  is taken in the group ring $\C[S_n]$.
The $c_k$ are known as   \emph{Young symmetrizers}.

\smallskip
\noindent
 The \emph{Specht module} $S^\nu$ may now be defined as the image  of the endomorphism of $ \C[S_n]$ that is
 right multiplication  by $c_k$ \cite[p. 119]{Fu}.
 The Specht module  $S^\nu$ is an irreducible $S_n$ module of dimension $$f^\nu=\frac{n!(n-2k+1)}{k!(n-k+1)!}=\tbinom{n}{k}-\tbinom{n}{k-1}$$ 
 \cite[Exercise 2.5.4]{Pr}, \cite[p. 88]{Fu}. The $S^\nu$ for distinct $\nu$'s are non-isomorphic.
 
  \begin{proposition}[Young's rule]\cite[p.\ 92]{Fu}, \cite[Thm.\ 3.3.1, Exercise 3.3.5]{Pr}.\label{Prop 1}\\
   Let $0\le m \le  \lfloor n/2 \rfloor$. Then
                                                 $$  M^{(n-m,m)}  \cong \bigoplus_{ k=0}^m  S^{(n-k,k)}.$$                                               
\qed
\end{proposition}

\section{Schur's lemma and Eigenvalues} 
\subsection{Diagonalizing}Let $0\le m\le \lfloor n/2 \rfloor$.
Choose a basis $d=d_1,\cdots,d_{\binom{n}{m}}$ of $M^{(n-m,m)}$ which is the union of bases of the $m+1$ irreducible submodules.
Then the basis $d$ diagonalizes all intertwining maps $M^{(n-m,m)}\to M^{(n-m,m)}$  simultaneously, by Schur's Lemma.
As recalled after Lemma~\ref{L5} below, the $d_i$ may be chosen in the $\Q$-span of the standard basis.
In particular, $B_{(n-m,m)}$ transforms to a diagonal matrix $D$ with $\Q$-linear combinations of the $b_i$ on the diagonal.
If one specializes $b_i=1$ and puts the other $b_j$ equal to zero, then the eigenvalues become algebraic integers because they
are roots of the characteristic polynomial of a matrix with integer entries.
We have proved:

\begin{proposition} The matrix $B_{(n-m,m)}$ has the properties:
\begin{itemize}
\item The eigenspaces  are independent of a (generic) choice of $b_0,\cdots b_m$,
\item The eigenvalues are $\mathbb Z$-linear combinations of  $b_0,\cdots b_m$.
\end{itemize}\qed
\end{proposition}

\subsection{Mapping Specht modules to a subset representation}
Let $\Omega$ be the last element of $\mathcal{C}^{n}_{m}$. So $\Omega=\llbracket n-m+1,n\rrbracket $, the set of integers in $[n-m+1,n ]$.
We define an $S_n$-linear map $\pi:\C[S_n]\to M^{(n-m,m)}$ by $$\pi(p)=p(\Omega).$$

\noindent
Our strategy is now as follows. We know already the eigenspaces of $B = B_{(n-m,m)}$ and want to compute eigenvalues. 
Below we take the eigenvector $\pi(c_k)$ and compare it with its image $B\pi(c_k)$. It is sufficient to consider only one of the coordinates,
in fact the `last coordinate' will do.   Note that $c_k$ is a double sum  of signed products $\pm qp$. We first look at the effect of $\pi$ on each term.

\medskip
\noindent
Let $0\le k\le m$. We will focus on the sets $\Omega\cap \pi(qp)$, where 
  $p\in R(T_k)$, $q\in C(T_k)$.
  Notice that $p$, $q$ permute the elements of $\llbracket 1,n\rrbracket $, not the boxes in a tableau. Nevertheless a
 diagram-figure makes it easier to follow the actions of $p$ and $q$ . See figure \ref{tab}. 
 We write $\Omega = \Omega_1 \cup \Omega_2$, where $\Omega_1=\llbracket n-m+1,n-k\rrbracket $ and $\Omega_2=\llbracket n-k+1,n\rrbracket $.

\begin{figure}[ht]
\caption{Positions of $V$ and $W$ in Young diagram of shape $(n-k,k)$.}\label{tab}
\unitlength.002em
\begin{center}
\begin{picture}(10878,5800)
\ifx\allinethickness\undefined
  \def\XFigeepicthickness#1{\relax}
\else
  \let\XFigeepicthickness\allinethickness
\fi
{\color{black}
\put(00,3055){\makebox(0,0)[lb]{\smash{\fontsize{12}{14.4}\usefont{T1}{ptm}{m}{n}1}}}
\put(10087,3165){\makebox(0,0)[lb]{\smash{\fontsize{12}{14.4}\usefont{T1}{ptm}{m}{n}n-k}}}
\color[gray]{0.75}
\polygon*(12,1515)(1737,1515)(1737,315)(12,315)
\color{black}
\linethickness{7.5\unitlength}\XFigeepicthickness{7.5\unitlength}
\polygon(12,1515)(1737,1515)(1737,315)(12,315)
\color[gray]{0.75}
\polygon*(8412,1515)(10812,1515)(10812,2865)(8412,2865)
\color{black}
\polygon(8412,1515)(10812,1515)(10812,2865)(8412,2865)
\put(612,805){\makebox(0,0)[lb]{\smash{\fontsize{12}{14.4}\usefont{T1}{ptm}{m}{n}$\Omega_2$}}}
\put(9362,2105){\makebox(0,0)[lb]{\smash{\fontsize{12}{14.4}\usefont{T1}{ptm}{m}{n}$\Omega_1$}}}
\put(652,3940){\makebox(0,0)[lb]{\smash{\fontsize{12}{14.4}\usefont{T1}{ptm}{m}{n}\Huge$\hookdownarrow$}}}
\put(802,5065){\makebox(0,0)[lb]{\smash{\fontsize{12}{14.4}\usefont{T1}{ptm}{m}{n}V}}}
\put(9362,5065){\makebox(0,0)[lb]{\smash{\fontsize{12}{14.4}\usefont{T1}{ptm}{m}{n}W}}}
\put(9362,3940){\makebox(0,0)[lb]{\smash{\fontsize{12}{14.4}\usefont{T1}{ptm}{m}{n}\Huge$\hookdownarrow$}}}
\put(1450,-100){\makebox(0,0)[lb]{\smash{\fontsize{12}{14.4}\usefont{T1}{ptm}{m}{n}n}}}
\put(1450,3090){\makebox(0,0)[lb]{\smash{\fontsize{12}{14.4}\usefont{T1}{ptm}{m}{n}k}}}
\put(7500,3165){\makebox(0,0)[lb]{\smash{\fontsize{12}{14.4}\usefont{T1}{ptm}{m}{n}n-m}}}
\polygon(12,2865)(8412,2865)(8412,1515)(12,1515)
\multiput(1737,1515)(0,60){23}{\line(0,1){35}}
}%
\end{picture}%
\end{center}
\end{figure}
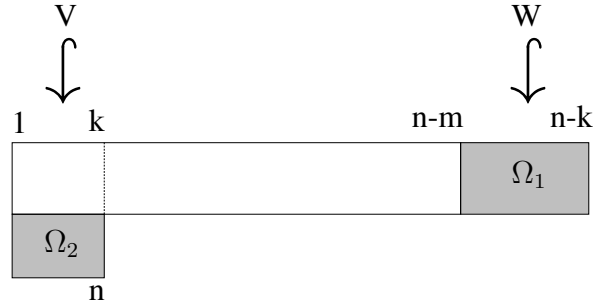

We put
\begin{equation}\label{V}
V=\{\;i\mid q(i)\neq i\leq k\;\}   \quad  \mbox{\rm and}
\end{equation}
\begin{equation}\label{W}
W=\{\;i\in \Omega\mid p(i)\notin\Omega\cup V \;\} \subseteq \Omega_1. 
\end{equation}
Notice that $\sgn(q)=(-1)^{\#V}$, where $\#X$ denotes the cardinality of a set $X$.

\begin{lemma}With these $q$, $p$, $V$, $W$, the cardinality of
$\Omega\cap \pi(qp)$ equals  $m-\#V-\# W$.\qed
\end{lemma}

\begin{lemma}Given  $V\subseteq\llbracket 1,k\rrbracket $, there is a unique $q\in C(T_k)$ such that equation \ref{V} holds.
\qed
\end{lemma}

\begin{lemma}Given  $V\subseteq\llbracket 1,k\rrbracket $, $W\subseteq \Omega_1 $, there are 
$$\binom{n-m-\#V}{\#W}(\#W)!\binom{m-k+\#V}{m-k-\#W}(m-k-\#W)!k!(n-m)!$$ 
elements $p$ in $R(T_k)$ such that equation \ref{W} holds.
\end{lemma}

\begin{proof} 
There are $\binom{n-m-\#V}{\#W}(\#W)!$ possibilities for the restriction of $p$ to $W$. There are $\binom{m-k+\#V}{m-k-\#W}(m-k-\#W)!$
possibilities for the restriction of $p$ to $\{\;i\in\Omega_1\backslash W \;\}$. 
Given both restrictions of $p$  there are still $k!(n-m)!$
possibilities.\qed
\end{proof}

Whenever we refer to `the last coordinate', this will be with respect to the standard basis. The last coordinate is the $\Omega$ coordinate.

\begin{lemma}\label{L5}The last coordinate of $\pi(c_k)$ is $(m-k)!k!(n-m)!$.
\end{lemma}

\begin{proof} 
We must take $V$ and $W$ empty.\qed
\end{proof}

In particular, the last coordinate of $\pi(c_k)$ is nonzero. This means that $\pi$ maps $S^{(n-k,k)}$ isomorphically into $M^{(n-m,m)}$.
By Schur's lemma $\pi(c_k)$ is an eigenvector of our universal intertwining matrix $B_{(n-m,m)}$. Notice that the $S_n$ orbits of the
$\pi(c_k)$, $k=0,\cdots, m$, together span all of $M^{(n-m,m)}$. So we may assume our diagonalizing basis $d$ is contained in the union
of these orbits. Then the $d_i$ are $\Q$-linear combinations of the standard basis. 

\begin{lemma}The last coordinate of $B_{(n-m,m)}\pi(c_k)$ is 
$$\sum _{v=0}^{k} \sum _{w=0}^{m-k} (-1)^{v} \binom{k}{v} \binom{m-k}{w} \binom{n-m-v}{w}w!\binom{m-k+v}{m-k-w}(m-k-w)!k!(n-m)!
  b_{m-v-w} .$$
\end{lemma}
\begin{proof} 
Multiply the last row of the matrix $B_{(n-m,m)}$ by $\pi (c_k)$. The result is the
the sum over all choices of $V\subseteq\llbracket 1,k\rrbracket $, $W\subseteq \Omega_1 $, where $v=\#V$ and $w=\#W$.\qed
\end{proof}

\begin{theorem}\label{Thm 1}The eigenvalue associated with the eigenvector $\pi(c_k)$ of $B_{n-m,m}$ is
$$ \lambda_k =\sum _{j=0}^k \sum _{p=0}^{m-k} (-1)^{k-j} \binom{k}{j} \binom{m-j}{p}  \binom{n-m-k+j}{m-k-p} b_{j+p}$$
and its multiplicity is $f^{(n-k,k)}=\frac{n!(n-2k+1)}{k!(n-k+1)!}$.
\end{theorem}
\begin{proof} 
Divide the last coordinate of $B_{(m,m)}\pi(c_k)$ by the last coordinate of $\pi(c_k)$.
Then rewrite, using the substitutions $v\mapsto k-j$, $w\mapsto m-k-p$.\qed
\end{proof}

\begin{example}
$B_{3,3}$ has eigenvalues:
\begin{itemize}
\item $ b_0 + 9 b_1 + 9 b_2 +b_3$ with multiplicity 1,
\item $ b_0 - b_1 - b_2 + b_3$ with multiplicity 9,
\item  $ -b_0 + 3 b_1 - 3 b_2 + b_3$  with multiplicity 5 and
\item $-b_0 -3 b_1 + 3b_2 + b_3$  with multiplicity 5. 
\end{itemize}
\end{example}

\subsection{Eberlein polynomials} The Eberlein polynomial $E_k$ is defined \cite[(4.33)]{De1} as
$$E_k(u)=\sum_{j=0}^{k}(-1)^{k-j} \binom{m-j}{k-j}\binom{m-u}{j}\binom{n-m+j-u}{j},$$
where $0\leq k\leq m\leq n/2$ as above. It is of degree $k$ in the variable $u(n+1-u)$.

There are several more descriptions of $E_k$ in \cite{De2}.
We now get:
\begin{corollary}
With $k,m,n$ as above, one has 
$$E_t(k)=\sum _{j=0}^k  (-1)^{k-j} \binom{k}{j} \binom{m-j}{m-t-j}  \binom{n-m-k+j}{n-m-t}$$
for $0\leq t\leq m$.
\end{corollary}
\begin{proof}By \cite[Thm 4.6]{De1} one may view $E_t(k)$ as the coefficient of $b[m-t]$ in $\lambda_k$ of our Theorem~\ref{Thm 1}.\qed
\end{proof}
\subsection{The Eisenbud-Levine-Khimshiashvili index computation}
We now turn to the problem that motivated the present work.

\begin{proposition}{\rm(\cite[Prop.\ 4]{Si})}
Substitute $$b_i=(-1)^i \frac{(2m-2i-1)!!(2i)!!}{(2m-1)!!},\qquad b_m=(-1)^m\frac{(2m)!!}{(2m-1)!!}$$
into $B_{m,m}$. The eigenvalues are $$\lambda_k=(-1)^m\frac{2m+1}{2m-2k+1}$$ for $0\le k\le m$,
with multiplicity $f^{(2m-k,k)}=\frac{(2m)!(2m-2k+1)}{k!(2m-k+1)!}$.
\end{proposition}

\begin{proof}
Plugging these values of $b_i$ into the formula in the theorem, one ends up with a multisum with a complicated hypergeometric summand.
We need to evaluate this multisum. Numerical experiments suggested the answer.
We now use the computer algebra package \cite{weg} that aims to give hints for proving a guessed answer. 
See the appendix to \cite{Si} for details. Or see  the Mathematica notebook that we attach to the arXiv-version of this paper.\qed
\end{proof}

\section*{Acknowledgements}
We thank the referee of  \cite{Si} for putting us on the right road. 
We thank Sam Mattheus for pointing out the work of Delsarte on Johnson schemes.
We also thank the Mathematical Institute of  Utrecht University
for offering a workspace to us emeriti.

This research did not receive any specific grant from funding agencies in the public, commercial, or not-for-profit sectors.


\begin{thebibliography}{99}
\bibitem{De1}Delsarte, P.,
\emph{An algebraic approach to the association schemes of coding theory}, 
Philips Res. Rep. Suppl. 1973, no. 10, vi+97 pp. 
\href{https://www.pearl-hifi.com/06_Lit_Archive/02_PEARL_Arch/Vol_16/Sec_53/Philips_Rsrch_Report_Supps_1961_thru_1976/Philips%20Research%20Reports%20Supplements%201973-10.pdf}{pdf}

\bibitem{De2}Delsarte, P.,
\emph{Properties and Applications of the Recurrence \\$F(i + 1, k + 1, n + 1) = q^{k + 1}F(i, k + 1, n)
- q^kF(i, k, n)$},\\
SIAM Journal on Applied Mathematics ,  1976, Vol. 31, 
pp. 262-270. \href{https://www.jstor.org/stable/2100244}{stable url at jstor}.

\bibitem{Fu} W. Fulton, \emph{Young Tableaux,} London Mathematical Society Student Texts 35, Cambridge University Press, 1997.
\bibitem{Pr}
Prasad, A., \emph{Representation Theory: A Combinatorial Viewpoint}, Cambridge Studies in Advanced Mathematics, 
Cambridge University Press, 2015.

\bibitem{A=B}
Petkov\v sek, M., Wilf, H., Zeilberger, D.,
$A=B$. 
With a foreword by Donald E. Knuth,  A K Peters, Ltd., Wellesley, MA, 1996. xii+212 pp. ISBN: 1-56881-063-6 05-01.
 \href{https://www.math.upenn.edu/~wilf/AeqB.html}{Homepage for the book $A=B$}.

\bibitem{Si}
Siersma, D., \emph{Extremal Area of Polygon Sliding along a Circle}, \href{https://arxiv.org/abs/2001.10882}{arXiv:2001.10882}.
Accepted by Hokkaido Mathematical Journal.


\bibitem{weg}
Wegschaider, K., \emph {Package MultiSum version 2.3} written by Kurt Wegschaider,
enhanced by Axel Riese and Burkhard Zimmermann,
Copyright Research Institute for Symbolic Computation (RISC),
Johannes Kepler University, Linz, Austria.


\end{thebibliography}
\end{document}